\documentclass[letterpaper, 10 pt, conference]{ieeeconf}  
\usepackage{color}
\usepackage{amsmath}
\usepackage{amssymb}
\usepackage{graphicx}
\usepackage{subfigure} 
\usepackage{comment,xspace}
\usepackage{algorithm,algorithmic}
\usepackage{fancybox}

\newtheorem{theorem}{Theorem}[section]
\newtheorem{proposition}{Proposition}

\newtheorem{remark}{Remark}[section]
\newtheorem{example}{Example}[section]

\newtheorem{corollary}{Corollary}[section]

\IEEEoverridecommandlockouts                              
\usepackage{cite}
\usepackage{graphics} 
\usepackage{amsmath} 
\usepackage{amssymb}  
\usepackage{subfigure}

\usepackage{scalefnt}

\newcommand{\mymargin}[1]{\marginpar{\color{red}\tiny\ttfamily#1}}
\renewcommand{\mymargin}[1]{}

\title{\LARGE \bf On the Difficulty of Deciding Asymptotic Stability of\\ Cubic Homogeneous Vector Fields}

\author{Amir Ali Ahmadi
\thanks{Amir Ali Ahmadi is with the Laboratory for Computer Science and Artificial Intelligence, Department of Electrical Engineering and
Computer Science, Massachusetts Institute of Technology. E-mail:
\texttt{a\_a\_a@mit.edu}.  }
 }

\begin{document}

\maketitle
\thispagestyle{empty}
\pagestyle{empty}

\begin{abstract}
It is well-known that asymptotic stability (AS) of homogeneous
polynomial vector fields of degree one (i.e., linear systems) can
be decided in polynomial time e.g. by searching for a quadratic
Lyapunov function. Since homogeneous vector fields of even degree
can never be AS, the next interesting degree to consider is equal
to three. In this paper, we prove that deciding AS of homogeneous
cubic vector fields is strongly NP-hard and pose the question of
determining whether it is even decidable. As a byproduct of the
reduction that establishes our NP-hardness result, we obtain a
Lyapunov-inspired technique for proving positivity of forms. We
also show that for asymptotically stable homogeneous cubic vector
fields in as few as two variables, the minimum degree of a
polynomial Lyapunov function can be arbitrarily large. Finally, we
show that there is no monotonicity in the degree of polynomial
Lyapunov functions that prove AS; i.e., a homogeneous cubic vector
field with no homogeneous polynomial Lyapunov function of some
degree $d$ can very well have a homogeneous polynomial Lyapunov
function of degree less than $d$.
\end{abstract}

\section{Introduction}
\subsection{Background}
We are concerned in this paper with a continuous time dynamical
system
\begin{equation}\label{eq:CT.dynamics}
\dot{x}=f(x),
\end{equation}
where $f:\mathbb{R}^n\rightarrow\mathbb{R}^n$ is a polynomial and
has an equilibrium at the origin, i.e., $f(0)=0$. Polynomial
differential equations appear ubiquitously in engineering and
sciences either as true models of physical systems, or as
approximations to other families of nonlinear dynamics. The
problem of deciding stability of equilibrium points of such
systems is of fundamental importance in control theory. The goal
of this paper is to demonstrate some of the difficulties
associated with answering stability questions about polynomial
vector fields in terms of both computational complexity and
non-existence of ``simple'' Lyapunov functions, even if one limits
attention to very restricted settings.

The notion of stability of interest in this paper is (local or
global) \emph{asymptotic stability}. The origin of
(\ref{eq:CT.dynamics}) is said to be \emph{stable in the sense of
Lyapunov} if for every $\epsilon>0$, there exists a
$\delta=\delta(\epsilon)>0$ such that $$||x(0)||<\delta\
\Rightarrow ||x(t)||<\epsilon, \ \ \forall t\geq0.$$ We say that
the origin is asymptotically stable (AS) if it is stable in the
sense of Lyapunov and $\delta$ can be chosen such that
$$||x(0)||<\delta \ \Rightarrow \
\lim_{t\rightarrow\infty}x(t)=0.$$ The origin is \emph{globally
asymptotically stable} (GAS) if it is stable in the sense of
Lyapunov and $\forall x(0)\in\mathbb{R}^n$,
$\lim_{t\rightarrow\infty}x(t)=0.$

The degree of the vector field in (\ref{eq:CT.dynamics}) is
defined to be the largest degree of the components of $f$. Our
focus in this paper is on \emph{homogeneous} polynomial vector
fields. A scalar valued function
$p:\mathbb{R}^n\rightarrow\mathbb{R}$ is said to be homogeneous
(of degree $d$) if it satisfies $p(\lambda x)=\lambda^d p(x)$ for
all $x\in\mathbb{R}^n$ and all $\lambda\in\mathbb{R}$. A
homogeneous polynomial is also called a \emph{form}. All monomials
of a form share the same degree. We say that the vector field $f$
in (\ref{eq:CT.dynamics}) is homogeneous if all components of $f$
are forms of the same degree. Homogeneous systems are extensively
studied in the literature on nonlinear control; see
e.g.~\cite{Stability_homog_poly_ODE}, \cite{Stabilize_Homog},
\cite{homog.feedback}, \cite{Baillieul_Homog_geometry},
\cite{Cubic_Homog_Planar}, \cite{HomogHomog},
\cite{homog.systems}. Since our results are negative in nature,
their validity for homogeneous polynomial systems obviously also
implies their validity for all polynomial systems.

A basic fact about homogeneous vector fields is that for these
systems the notions of local and global asymptotic stability are
equivalent. Indeed, a homogeneous vector field of degree $d$
satisfies $f(\lambda x)=\lambda^d f(x)$ for any scalar $\lambda$,
and therefore the value of $f$ on the unit sphere determines its
value everywhere. It is also well-known that an asymptotically
stable homogeneous system admits a homogeneous Lyapunov
function~\cite[Sec. 57]{Hahn_stability_book},\cite{HomogHomog}.

\subsection{An open question of Arnold}

It is natural to ask whether stability of equilibrium points of
polynomial vector fields can be decided in finite time. In fact,
this is a well-known question of Arnold that appears
in~\cite{Arnold_Problems_for_Math}:

\vspace{-5pt}
\begin{itemize}
\item[] ``Is the stability problem for stationary points
algorithmically decidable? The well-known Lyapunov
theorem\footnote{The theorem that Arnold is referring to here is
the indirect method of Lyapunov related to linearization. This is
not to be confused with Lyapunov's direct method (or the second
method), which is what we are concerned with in sections that
follow.} solves the problem in the absence of eigenvalues with
zero real parts. In more complicated cases, where the stability
depends on higher order terms in the Taylor series, there exists
no algebraic criterion.

Let a vector field be given by polynomials of a fixed degree, with
rational coefficients. Does an algorithm exist, allowing to
decide, whether the stationary point is stable?''
\end{itemize}
\vspace{-5pt}

%
%

To our knowledge, there has been no formal resolution to this
question, neither for the case of stability in the sense of
Lyapunov, nor for the case of asymptotic stability (in its local
or global version). In~\cite{Costa_Doria_undecidabiliy}, da Costa
and Doria show that if the right hand side of the differential
equation contains elementary functions (sines, cosines,
exponentials, absolute value function, etc.), then there is no
algorithm for deciding whether the origin is stable or unstable.
They also present a dynamical system in~\cite{Costa_Doria_Hopf}
where one cannot decide whether a Hopf bifurcation will occur or
whether there will be parameter values such that a stable fixed
point becomes unstable.
%
%
A relatively larger number of undecidability results are available
for questions related to other properties of polynomial vector
fields, such as
reachability~\cite{Undecidability_vec_fields_survey} or
boundedness of domain of
definition~\cite{Bounded_Defined_ODE_Undecidable}, or for
questions about stability of hybrid systems~\cite{TsiLinSat},
\cite{BlTi2}, \cite{BlTi_stab_contr_hybrid},
\cite{Deciding_stab_mortal_PWA}. We refer the interested reader to
the survey papers in~\cite{Survey_CT_Computation},
\cite{Undecidability_vec_fields_survey},
\cite{Sontag_complexity_comparison},
\cite{BlTi_complexity_3classes}, \cite{BlTi1}.


We are also interested to know whether the problem of deciding
asymptotic stability of homogeneous polynomial vector fields is
undecidable for some fixed degree, say, equal to $3$. The answer
to such decidability questions, or at least the level of
difficulty associated with proving such results, can depend in a
subtle way on the exact criteria in question. For example, it has
been known for a while that the question of determining
boundedness of trajectories for arbitrarily switched linear
systems is undecidable~\cite{BlTi2} even when one restricts
attention to switched systems defined by nonnegative matrices. On
the other hand, the complexity of testing asymptotic stability for
the same class of systems remains open and in fact is conjectured
to be decidable~\cite{Jungers_Blondel_FP_Rational}.

\subsection{Existence of polynomial Lyapunov functions}

For stability analysis of polynomial vector fields, it is most
common (and quite natural) to search for Lyapunov functions that
are polynomials themselves. This approach has become further
prevalent over the past decade due to the fact that techniques
from sum of squares optimization~\cite{PhD:Parrilo} have provided
for algorithms that given a polynomial system can efficiently
search for a polynomial Lyapunov function~\cite{PhD:Parrilo},
\cite{PapP02}. The question is therefore naturally motivated to
determine whether stable polynomial systems always admit
polynomial Lyapunov functions, and whether one can give upper
bounds on the degree of such Lyapunov functions in cases when they
do exist. A study of questions of this type for different notions
of stability has recently been carried out
in~\cite{Peet.exp.stability},
\cite{Peet.Antonis.converse.sos.journal},
\cite{AAA_MK_PP_CDC11_no_Poly_Lyap},
\cite{AAA_PP_CDC11_converseSOS_Lyap}, \cite[Chap. 4]{AAA_PhD}. In
this paper, we continue this line of research by studying the case
where the vector field is homogeneous.

Throughout this paper, by a (polynomial) Lyapunov function for
(\ref{eq:CT.dynamics}), we mean a positive definite polynomial
function $V$ whose derivatives $\dot{V}$ along trajectories of
(\ref{eq:CT.dynamics}) is negative definite; i.e., a function $V$
satisfying
\begin{eqnarray}
V(x)&>&0\quad \forall x\neq0 \label{eq:V.positive} \\
\dot{V}(x)=\langle\nabla V(x),f(x)\rangle&<&0\quad \forall x\neq0.
\label{eq:Vdot.negative}
\end{eqnarray}
Here, $\nabla V(x)$ denotes the gradient vector of $V$, and
$\langle .,. \rangle$ is the standard inner product in
$\mathbb{R}^n$. If such a $V$ is also radially unbounded, then the
inequalities in (\ref{eq:V.positive}) and (\ref{eq:Vdot.negative})
imply that the origin of (\ref{eq:CT.dynamics}) is GAS. When the
dynamics $f$ is homogeneous, we can restrict our search to
homogeneous polynomials. Such a Lyapunov function is automatically
radially unbounded and proves (local or equivalently global)
asymptotic stability of the homogeneous vector field.

Naturally, questions regarding complexity of deciding asymptotic
stability and questions about existence of Lyapunov functions are
related. For instance, if one proves that for a class of
polynomial vector fields, asymptotic stability implies existence
of a polynomial Lyapunov function together with a computable upper
bound on its degree, then the question of asymptotic stability for
that class becomes decidable. This is due to the fact that given a
polynomial system and an integer $d$, the question of deciding
whether the system admits a polynomial Lyapunov function of degree
$d$ can be answered in finite time using quantifier
elimination~\cite{Tarski_quantifier_elim},~\cite{Seidenberg_quantifier_elim}.

For the case of linear systems (i.e., homogeneous systems of
degree $1$), the situation is particularly nice. If such a system
is asymptotically stable, then there always exists a
\emph{quadratic} Lyapunov function. Asymptotic stability of a
linear system $\dot{x}=Ax$ is equivalent to the easily checkable
algebraic criterion that the eigenvalues of $A$ be in the open
left half complex plane. Deciding this property of the matrix $A$
can formally be done in polynomial time, e.g. by solving a
Lyapunov equation~\cite{BlTi1}.

Moving up in the degree, it is not difficult to show that if a
homogeneous polynomial vector field has even degree, then it can
never be asymptotically stable; see e.g.~\cite[p.
283]{Hahn_stability_book}. So the next interesting case occurs for
homogeneous vector fields of degree $3$. We will prove three
results in this paper which demonstrate that already for cubic
homogeneous systems, the situation is significantly more complex
than it is for linear systems. We outline our contributions next.

\subsection{Contributions and organization of this paper}
In Section~\ref{sec:NP-hard.proof}, we prove that determining
asymptotic stability for homogeneous cubic vector fields is
strongly NP-hard (Theorem~\ref{thm:asym.stability.nphard}).
Although this of course does not resolve the question of Arnold,
the result gives a lower bound on the complexity of this problem.
It is an interesting open question to investigate whether in this
specific setting, the problem is also undecidable.

The implication of the NP-hardness of this problem is that unless
P=NP, it is impossible to design an algorithm that can take as
input the (rational) coefficients of a homogeneous cubic vector
field, have running time bounded by a polynomial in the number of
bits needed to represent these coefficients, and always output the
correct yes/no answer on asymptotic stability. Moreover, the fact
that our NP-hardness result is in the strong sense (as opposed to
weakly NP-hard problems such as KNAPSACK, SUBSET SUM, etc.)
implies that the problem remains NP-hard even if the size (bit
length) of the coefficients is $O(\log n)$, where $n$ is the
dimension. For a strongly NP-hard problem, even a
pseudo-polynomial time algorithm cannot exist unless P=NP.
See~\cite{GareyJohnson_Book} for precise definitions and more
details.

In Section~\ref{sec:NP-hard.proof}, we also present a
Lyapunov-inspired technique for proving positivity of forms that
comes directly out of the reduction in the proof of our
NP-hardness result (Corollary~\ref{cor:positivity.forms.Lyap}). We
show the potential advantages of this technique over standard sum
of squares techniques on an example
(Example~\ref{ex:form.pos.Lyap.inspired}).


In Section~\ref{sec:no.finite.bound}, we prove that unlike AS
linear systems that always admit quadratic Lyapunov functions, AS
cubic homogeneous systems may need polynomial Lyapunov functions
of arbitrarily large degree, even when the dimension is fixed to
$2$ (Theorem~\ref{thm:no.finite.bound}). Finally, in
Section~\ref{sec:no.monotonicity.in.degree}, we show that there is
no monotonicity in the degree of homogeneous polynomial Lyapunov
functions for homogeneous cubic vector fields. We give an example
of such a vector field which admits a homogeneous polynomial
Lyapunov function of degree $4$ but not one of degree $6$
(Theorem~\ref{thm:no.monotonicity}).

\section{NP-hardness of deciding asymptotic stability of homogeneous cubic vector
fields}\label{sec:NP-hard.proof}




The main result of this section is the following theorem.

\begin{theorem}\label{thm:asym.stability.nphard}
Deciding asymptotic stability of homogeneous cubic polynomial
vector fields is strongly NP-hard.
\end{theorem}

The key idea behind the proof of this theorem is the following: We
will relate the solution of a combinatorial problem not to the
behavior of the trajectories of a cubic vector field that are hard
to get a handle on, but instead to properties of a Lyapunov
function that proves asymptotic stability of this vector field. As
we will see shortly, insights from Lyapunov theory make the proof
of this theorem quite simple. The reduction is broken into two
steps:
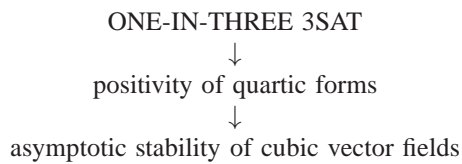
\begin{center}
ONE-IN-THREE 3SAT \\ $\downarrow$ \\ positivity of quartic forms
\\ $\downarrow$\\  asymptotic stability of cubic vector fields
\end{center}


\subsection{Reduction from ONE-IN-THREE 3SAT to positivity of quartic forms}

A form $q$ is said to be \emph{nonnegative} or \emph{positive
semidefinite} if $q(x)\geq 0$ for all $x$ in $\mathbb{R}^n$. We
say that a form $q$ is \emph{positive definite} if $q(x)>0$ for
all $x\neq 0$ in $\mathbb{R}^n$. (Note that forms necessarily
vanish at the origin.) It is well-known that deciding
nonnegativity of quartic forms is NP-hard; see
e.g.~\cite{nonnegativity_NP_hard}
and~\cite{deKlerk_StableSet_copositive}.
%
%
For reasons that will become clear shortly, we are interested
instead in showing hardness of deciding positive definiteness of
quartic forms. This is in some sense even easier to accomplish. A
very straightforward reduction from 3SAT proves NP-hardness of
deciding positive definiteness of polynomials of degree $6$. By
using ONE-IN-THREE 3SAT instead, we will reduce the degree of the
polynomial from $6$ to $4$.

\begin{proposition}\label{prop:positivity.quartic.NPhard}
It is strongly\footnote{The NP-hardness results of this section
will all be in the strong sense. From here on, we drop the prefix
``strong'' for brevity.} NP-hard to decide whether a homogeneous
polynomial of degree $4$ is positive definite.
\end{proposition}

\begin{proof}
We give a reduction from ONE-IN-THREE 3SAT which is known to be
NP-complete~\cite[p. 259]{GareyJohnson_Book}. Recall that in
ONE-IN-THREE 3SAT, we are given a 3SAT instance (i.e., a
collection of clauses, where each clause consists of exactly three
literals, and each literal is either a variable or its negation)
and we are asked to decide whether there exists a $\{0,1\}$
assignment to the variables that makes the expression true with
the additional property that each clause has \emph{exactly one}
true literal.

To avoid introducing unnecessary notation, we present the
reduction on a specific instance. The pattern will make it obvious
that the general construction is no different. Given an instance
of ONE-IN-THREE 3SAT, such as the following
\begin{equation}\label{eq:reduciton.1-in-3.3sat.instance}
(x_1\vee\bar{x}_2\vee x_4)\wedge (\bar{x}_2\vee\bar{x}_3\vee
x_5)\wedge (\bar{x}_1\vee x_3\vee \bar{x}_5)\wedge (x_1\vee
x_3\vee x_4),
\end{equation}
we define the quartic polynomial $p$ as follows:
\begin{equation}\label{eq:reduction.p}
\begin{array}{lll}
p(x)&=&\sum_{i=1}^5 x_i^2(1-x_i)^2\\ \ &\
&+(x_1+(1-x_2)+x_4-1)^2+((1-x_2)\\ \ &\ &+(1-x_3)+x_5-1)^2
\\ \ &\ &+((1-x_1)+x_3+(1-x_5)-1)^2\\ \ &\
&+(x_1+x_3+x_4-1)^2.
\end{array}
\end{equation}
Having done so, our claim is that $p(x)>0$ for all $x\in
\mathbb{R}^5$ (or generally for all $x\in \mathbb{R}^n$) if and
only if the ONE-IN-THREE 3SAT instance is not satisfiable. Note
that $p$ is a sum of squares and therefore nonnegative. The only
possible locations for zeros of $p$ are by construction among the
points in $\{0,1\}^5$. If there is a satisfying Boolean assignment
$x$ to (\ref{eq:reduciton.1-in-3.3sat.instance}) with exactly one
true literal per clause, then $p$ will vanish at point $x$.
Conversely, if there are no such satisfying assignments, then for
any point in $\{0,1\}^5$, at least one of the terms in
(\ref{eq:reduction.p}) will be positive and hence $p$ will have no
zeros.

It remains to make $p$ homogeneous. This can be done via
introducing a new scalar variable $y$. If we let
\begin{equation}\label{eq:reduction.ph}
p_h(x,y)=y^4 p(\textstyle{\frac{x}{y}}),
\end{equation}
then we claim that $p_h$ (which is a quartic form) is positive
definite if and only if $p$ constructed as in
(\ref{eq:reduction.p}) has no zeros.\footnote{In general, the
homogenization operation in (\ref{eq:reduction.ph}) does not
preserve positivity. For example, as shown in~\cite{Reznick}, the
polynomial $x_1^2+(1-x_1x_2)^2$ has no zeros, but its
homogenization $x_1^2y^2+(y^2-x_1x_2)^2$ has zeros at the points
$(1,0,0)^T$ and $(0,1,0)^T$. Nevertheless, positivity is preserved
under homogenization for the special class of polynomials
constructed in this reduction, essentially because polynomials of
type (\ref{eq:reduction.p}) have no zeros at infinity.} Indeed, if
$p$ has a zero at a point $x$, then that zero is inherited by
$p_h$ at the point $(x,1)$. If $p$ has no zeros, then
(\ref{eq:reduction.ph}) shows that $p_h$ can only possibly have
zeros at points with $y=0$. However, from the structure of $p$ in
(\ref{eq:reduction.p}) we see that
$$p_h(x,0)=x_1^4+\cdots+x_5^4,$$ which cannot be zero (except at
the origin). This concludes the proof.
\end{proof}

\subsection{Reduction from positivity of quartic forms to asymptotic stability of cubic vector fields}
We now present the second step of the reduction and finish the
proof of Theorem~\ref{thm:asym.stability.nphard}.

\begin{proof}[Proof of Theorem~\ref{thm:asym.stability.nphard}]
We give a reduction from the problem of deciding positive
definiteness of quartic forms, whose NP-hardness was established
in Proposition~\ref{prop:positivity.quartic.NPhard}. Given a
quartic form $V\mathrel{\mathop:}=V(x)$, we define the polynomial
vector field
\begin{equation}\label{eq:xdot.cubic.reduction}
\dot{x}=-\nabla V(x).
\end{equation}
Note that the vector field is homogeneous of degree $3$. We claim
that the above vector field is (locally or equivalently globally)
asymptotically stable if and only if $V$ is positive definite.
First, we observe that by construction
\begin{equation}\label{eq:Vdot<=0.always}
\dot{V}(x)=\langle \nabla V(x), \dot{x} \rangle=-||\nabla
V(x)||^2\leq 0.
\end{equation}
Suppose $V$ is positive definite. By Euler's identity for
homogeneous functions,\footnote{Euler's identity is easily derived
by differentiating both sides of the equation $V(\lambda
x)~=~\lambda^d V(x)$ with respect to $\lambda$ and setting
$\lambda=1$.} we have $V(x)=\frac{1}{4}x^T\nabla V(x).$ Therefore,
positive definiteness of $V$ implies that $\nabla V(x)$ cannot
vanish anywhere except at the origin. Hence, $\dot{V}(x)<0$ for
all $x\neq 0$. In view of Lyapunov's theorem (see e.g.~\cite[p.
124]{Khalil:3rd.Ed}), and the already mentioned fact that a
positive definite homogeneous function is radially unbounded, it
follows that the system in (\ref{eq:xdot.cubic.reduction}) is
globally asymptotically stable.

For the converse direction, suppose
(\ref{eq:xdot.cubic.reduction}) is GAS. Our first claim is that
global asymptotic stability together with $\dot{V}(x)\leq 0$
implies that $V$ must be positive semidefinite. This follows from
the following simple argument, which we have also previously
presented in~\cite{AAA_PP_ACC11_Lyap_High_Deriv} for a different
purpose. Suppose for the sake of contradiction that for some
$\hat{x}\in\mathbb{R}^n$ and some $\epsilon>0,$ we had
$V(\hat{x})=-\epsilon<0$. Consider a trajectory $x(t;\hat{x})$ of
system (\ref{eq:xdot.cubic.reduction}) that starts at initial
condition $\hat{x}$, and let us evaluate the function $V$ on this
trajectory. Since $V(\hat{x})=-\epsilon$ and $\dot{V}(x)\leq 0$,
we have $V(x(t;\hat{x}))\leq-\epsilon$ for all $t>0$. However,
this contradicts the fact that by global asymptotic stability, the
trajectory must go to the origin, where $V$, being a form,
vanishes.

To prove that $V$ is positive definite, suppose by contradiction
that for some nonzero point $x^*\in\mathbb{R}^n$ we had
$V(x^*)=0$. Since we just proved that $V$ has to be positive
semidefinite, the point $x^*$ must be a global minimum of $V$.
Therefore, as a necessary condition of optimality, we should have
$\nabla V(x^*)=0$. But this contradicts the system in
(\ref{eq:xdot.cubic.reduction}) being GAS, since the trajectory
starting at $x^*$ stays there forever and can never go to the
origin.
\end{proof}

Perhaps of independent interest, the reduction we just gave
suggests a method for proving positive definiteness of forms.
Given a form $V$, we can construct a dynamical system as in
(\ref{eq:xdot.cubic.reduction}), and then any method that we may
have for proving stability of vector fields (e.g. the use of
various kinds of Lyapunov functions) can serve as an algorithm for
proving positivity of $V$. In particular, if we use a polynomial
Lyapunov function $W$ to prove stability of the system in
(\ref{eq:xdot.cubic.reduction}), we get the following corollary.

\begin{corollary}\label{cor:positivity.forms.Lyap}
Let $V$ and $W$ be two forms of possibly different degree. If $W$
is positive definite, and $\langle \nabla W, \nabla V \rangle$ is
positive definite, then $V$ is positive definite.
\end{corollary}

A polynomial $p$ is said to be a \emph{sum of squares} (sos) if it
can be written as $p=\sum_{i=1}^m q_i^2$ for some polynomials
$q_i$. An sos polynomial is clearly nonnegative. Moreover, unlike
the property of nonnegativity that is NP-hard to check, existence
of an sos decomposition can be cast as a semidefinite
program~\cite{sdprelax}, which can be solved efficiently. However,
not every nonnegative polynomial is a sum of squares.

An interesting fact about
Corollary~\ref{cor:positivity.forms.Lyap} is that its algebraic
version with sum of squares replaced for positivity is not true.
In other words, we can have $W$ sos (and positive definite),
$\langle \nabla W, \nabla V \rangle$ sos (and positive definite),
but $V$ not sos. This gives us a way of proving positivity of some
polynomials that are not sos, using \emph{only} sos certificates.
Given a form $V$, since the expression $\langle \nabla W, \nabla V
\rangle$ is linear in the coefficients of $W$, we can use
semidefinite programming to search for a form $W$ that satisfies
$W$ sos and $\langle \nabla W, \nabla V \rangle$ sos, and this
would prove positivity of $V$. The following example demonstrates
the potential usefulness of this approach.

\begin{example}\label{ex:form.pos.Lyap.inspired}
Consider the following form of degree $6$:
\begin{equation}\label{eq:V.positive.Motzkin}
V(x)=
x_1^4x_2^2+x_1^2x_2^4-3x_1^2x_2^2x_3^2+x_3^6+\frac{1}{250}(x_1^2+x_2^2+x_3^2)^3.
\end{equation}
One can check that this polynomial is not a sum of squares. (In
fact, this is the celebrated Motzkin form~\cite{MotzkinSOS}
slightly perturbed.) On the other hand, we can use the software
package YALMIP~\cite{yalmip} together with the SDP solver
SeDuMi~\cite{sedumi} to search for a form $W$ satisfying
\begin{equation}\label{eq:W.sos.gradW.gradV.sos}
\begin{array} {rl}
W & \mbox{sos} \\
\langle \nabla W, \nabla V \rangle & \mbox{sos.}
\end{array}
\end{equation}
If we parameterize $W$ as a quadratic form, no feasible solution
will be returned form the solver. However, when we increase the
degree of $W$ from $2$ to $4$, the solver returns the following
polynomial
\begin{equation}\nonumber
\begin{array}{lll}
W(x)&=&9x_2^4+9x_1^4-6x_1^2x_2^2+6x_1^2x_3^2+6x_2^2x_3^2+3x_3^4\\
\ &\ &-x_1^3x_2-x_1x_2^3-x_1^3x_3-3x_1^2x_2x_3-3x_1x_2^2x_3\\ \ &\
&-x_2^3x_3-4x_1x_2x_3^2-x_1x_3^3-x_2x_3^3
\end{array}
\end{equation}
that satisfies both sos constraints in
(\ref{eq:W.sos.gradW.gradV.sos}). One can easily infer from the
sos decompositions (e.g. by checking positive definiteness of the
associated ``Gram matrices'') that the forms $W$ and $\langle
\nabla W, \nabla V \rangle$ are positive definite. Hence, by
Corollary~\ref{cor:positivity.forms.Lyap}, we have a proof that
$V$ in (\ref{eq:V.positive.Motzkin}) is positive definite.
$\triangle$
\end{example}

Interestingly, approaches of this type that use gradient
information for proving positivity of polynomials with sum of
squares techniques have been studied by Nie, Demmel, and Sturmfels
in~\cite{Gradient_Ideal_SOS}, though the derivation there is not
inspired by Lyapunov theory.

\section{Non-existence of a uniform bound on the degree of polynomial Lyapunov functions in fixed dimension and
degree}\label{sec:no.finite.bound}

For polynomial vector fields in general, existence of a polynomial
Lyapunov function is not necessary for global asymptotic
stability. In joint work with M. Krstic and P.A.
Parrilo~\cite{AAA_MK_PP_CDC11_no_Poly_Lyap}, we recently gave a
remarkably simple example of a (non-homogeneous) quadratic
polynomial vector field in two variables that is GAS but does not
admit a polynomial Lyapunov function (of any degree). An
independent earlier example that appears in a book by Bacciotti
and Rosier~\cite[Prop. 5.2]{Bacciotti.Rosier.Liapunov.Book} was
brought to our attention after our work was submitted. We refer
the reader to~\cite{AAA_MK_PP_CDC11_no_Poly_Lyap} for a discussion
on the differences between the two examples, the main one being
that the example in~\cite{Bacciotti.Rosier.Liapunov.Book} does not
admit a polynomial Lyapunov function even locally but unlike the
example in~\cite{AAA_MK_PP_CDC11_no_Poly_Lyap} relies on using
irrational coefficients.

The situation for homogeneous polynomial vector fields, however,
seems to be different. We conjecture that for such systems,
existence of a homogeneous polynomial Lyapunov function is
necessary and sufficient for (global) asymptotic stability. The
reason for this conjecture is that we expect that one should be
able to approximate a continuously differentiable Lyapunov
function with a polynomial one on the unit sphere, which by
homogeneity should be enough to imply the Lyapunov inequalities
everywhere. A formal treatment of this idea is left for future
work. Here, we build on the result in~\cite[Prop.
5.2]{Bacciotti.Rosier.Liapunov.Book} to prove that the minimum
degree of a polynomial Lyapunov function for an AS homogeneous
vector field can be arbitrarily large even when the degree and
dimension are fixed respectively to $3$ and $2$.

\begin{proposition}[{~\cite[Prop.
5.2--a]{Bacciotti.Rosier.Liapunov.Book}}]
\label{prop:Bacciotti.Rosier} Consider the vector field
\begin{equation}\label{eq:Bacciotti.Rosier.f0}
\begin{array}{lll}
\dot{x}&=&-2\lambda y(x^2+y^2)-2y(2x^2+y^2) \\
\dot{y}&=&\ \ 4\lambda x(x^2+y^2)+2x(2x^2+y^2)
\end{array}
\end{equation}
parameterized by the scalar $\lambda>0$. For all values of
$\lambda$ the origin is a center for
(\ref{eq:Bacciotti.Rosier.f0}), but for any irrational value of
$\lambda$ there exists no polynomial function $V$ satisfying
$\dot{V}(x,y)=\frac{\partial{V}}{\partial{x}}\dot{x}+\frac{\partial{V}}{\partial{y}}\dot{y}=0.$
\end{proposition}

\begin{theorem}\label{thm:no.finite.bound}
Let $\lambda$ be a positive irrational real number and consider
the following homogeneous cubic vector field parameterized by the
scalar $\theta$: \scalefont{.82}
\begin{equation}\label{eq:a.s.cubic.vec.arbitrary.high.Lyap}
\begin{pmatrix}\dot{x} \\ \dot{y}\end{pmatrix} =\begin{pmatrix}\cos(\theta) & -\sin(\theta)\\\sin(\theta)&\ \ \  \cos(\theta)  \end{pmatrix} \begin{pmatrix}-2\lambda y(x^2+y^2)-2y(2x^2+y^2) \\
\ \ 4\lambda x(x^2+y^2)+2x(2x^2+y^2)
\end{pmatrix}.
\end{equation} \normalsize
Then for any even degree $d$ of a candidate polynomial Lyapunov
function, there exits a $\theta>0$ small enough such that the
vector field in (\ref{eq:a.s.cubic.vec.arbitrary.high.Lyap}) is
asymptotically stable but does not admit a polynomial Lyapunov
function of degree $\leq d$.
\end{theorem}

\begin{proof}
Consider the (non-polynomial) positive definite Lyapunov function
\begin{equation}\nonumber
V(x,y)=(2x^2+y^2)^\lambda(x^2+y^2)
\end{equation}
whose derivative along the trajectories of
(\ref{eq:a.s.cubic.vec.arbitrary.high.Lyap}) is equal to
$$\dot{V}(x,y)=-\sin(\theta)(2x^2+y^2)^{\lambda-1}(\dot{x}^2+\dot{y}^2).$$ Since
$\dot{V}$ is negative definite for $0<\theta<\pi$, it follows that
for $\theta$ in this range, the origin of
(\ref{eq:a.s.cubic.vec.arbitrary.high.Lyap}) is asymptotically
stable.

To establish the claim in the theorem, suppose for the sake of
contradiction that there exists an upper bound $\bar{d}$ such that
for all $0<\theta<\pi$ the system admits a (homogeneous)
polynomial Lyapunov function of degree $d(\theta)$ with
$d(\theta)\leq\bar{d}$. Let $\hat{d}$ be the least common
multiplier of the degrees $d(\theta)$ for $0<\theta<\pi$. (Note
that $d(\theta)$ can at most range over all even positive integers
less than or equal to $\bar{d}$.) Since positive powers of
Lyapunov functions are valid Lyapunov functions, it follows that
for every $0<\theta<\pi$, the system admits a homogeneous
polynomial Lyapunov function $W_\theta$ of degree $\hat{d}$. By
rescaling, we can assume without loss of generality that all
Lyapunov functions $W_\theta$ have unit area on the unit sphere.
Let us now consider the sequence $\{W_\theta\}$ as
$\theta\rightarrow 0$. We think of this sequence as residing in a
compact subset of $\mathbb{R}^{\hat{d}+1\choose \hat{d}}$
associated with the set $P_{2,\hat{d}}$ of (coefficients of) all
nonnegative bivariate homogeneous polynomials of degree $\hat{d}$
with unit area on the unit sphere. Since every bounded sequence
has a converging subsequence, it follows that there must exist a
subsequence of $\{W_\theta\}$ that converges (in the coefficient
sense) to some polynomial $W_0$ belonging to $ P_{2,\hat{d}}$.
Since convergence of this subsequence also implies convergence of
the associated gradient vectors, we get that
$$\dot{W}_0(x,y)=\frac{\partial{W}_0}{\partial{x}}\dot{x}+\frac{\partial{W}_0}{\partial{y}}\dot{y}\leq0.$$
On the other hand, when $\theta=0$, the vector field in
(\ref{eq:a.s.cubic.vec.arbitrary.high.Lyap}) is the same as the
one in (\ref{eq:Bacciotti.Rosier.f0}) and hence the trajectories
starting from any nonzero initial condition go on periodic orbits.
This however implies that $\dot{W}=0$ everywhere and in view of
Proposition~\ref{prop:Bacciotti.Rosier} we have a contradiction.
\end{proof}

\begin{remark}
Unlike the result in~\cite[Prop.
5.2]{Bacciotti.Rosier.Liapunov.Book}, it is easy to establish the
result of Theorem~\ref{thm:no.finite.bound} without having to use
irrational coefficients in the vector field. One approach is to
take an irrational number, e.g. $\pi$, and then think of a
sequence of vector fields given by
(\ref{eq:a.s.cubic.vec.arbitrary.high.Lyap}) that is parameterized
by both $\theta$ and $\lambda$. We let the $k$-th vector field in
the sequence have $\theta_k=\frac{1}{k}$ and $\lambda_k$ equal to
a rational number representing $\pi$ up to $k$ decimal digits.
Since in the limit as $k\rightarrow\infty$ we have
$\theta_k\rightarrow 0$ and $\lambda_k\rightarrow \pi$, it should
be clear from the proof of Theorem~\ref{thm:no.finite.bound} that
for any integer $d$, there exists an AS bivariate homogeneous
cubic vector field with \emph{rational} coefficients that does not
have a polynomial Lyapunov function of degree less than $d$.
\end{remark}

\section{Lack of monotonicity in the degree of polynomial Lyapunov
functions}\label{sec:no.monotonicity.in.degree} If a dynamical
system admits a quadratic Lyapunov function $V$, then it clearly
also admits a polynomial Lyapunov function of any higher even
degree (e.g. simply given by $V^k$ for $k=2,3,\ldots$). However,
our next theorem shows that for homogeneous systems that do not
admit a quadratic Lyapunov function, such a monotonicity property
in the degree of polynomial Lyapunov functions may not hold.

\begin{theorem}\label{thm:no.monotonicity}
Consider the following homogeneous cubic vector field
parameterized by the scalar $\theta$:
\begin{equation}\label{eq:non.monotonicity.vec.field}
\begin{pmatrix}\dot{x} \\ \dot{y}\end{pmatrix} =\begin{pmatrix}-\sin(\theta) & \ \ \cos(\theta)\\-\cos(\theta)&-\sin(\theta)  \end{pmatrix} \begin{pmatrix}x^3 \\
y^3
\end{pmatrix}.
\end{equation}
There exists a range of values for the parameter $\theta>0$ for
which the vector field is asymptotically stable, has no
homogeneous polynomial Lyapunov function of degree $6$, but admits
a homogeneous polynomial Lyapunov function of degree $4$.
\end{theorem}

\begin{proof}
Consider the positive definite Lyapunov function
\begin{equation}\label{eq:V=x^4+y^4}
V(x,y)=x^4+y^4.
\end{equation}
The derivative of this Lyapunov function is given by
$$\dot{V}(x,y)=-4\sin(\theta)(x^6+y^6),$$ which is negative definite for $0<\theta<\pi$.
Therefore, when $\theta$ belongs to this range, the origin of
(\ref{eq:a.s.cubic.vec.arbitrary.high.Lyap}) is asymptotically
stable and the system admits the degree $4$ Lyapunov function
given in (\ref{eq:V=x^4+y^4}). On the other hand, we claim that
for $\theta$ small enough, the system cannot admit a degree $6$
(homogeneous) polynomial Lyapunov function. To argue by
contradiction, we suppose that for arbitrarily small and positive
values of $\theta$ the system admits sextic Lyapunov functions
$W_\theta$. Since the vector field satisfies the symmetry
\begin{equation}\nonumber
\begin{pmatrix}\dot{x}(y,-x) \\ \dot{y}(y,-x)\end{pmatrix} =\begin{pmatrix}0 & 1\\-1&0  \end{pmatrix} \begin{pmatrix}\dot{x} \\
\dot{y}
\end{pmatrix},
\end{equation}
we can assume that the Lyapunov functions $W_\theta$ satisfy the
symmetry $W_\theta(y,-x)=W_\theta(x,y)$. \footnote{To see this,
note that any Lyapunov function $V_\theta$ for this system can be
made into one satisfying this symmetry by letting
$W_\theta(x,y)=V_\theta(x,y)+V_\theta(y,-x)+V_\theta(-x,-y)+V_\theta(-y,x)$.}
This means that $W_\theta$ can be parameterized with no odd
monomials, i.e., in the form
\begin{equation}\nonumber
W_\theta(x,y)=c_1x^6+c_2x^2y^4+c_3x^4y^2+c_4y^6,
\end{equation}
where it is understood that the coefficients $c_1,\ldots,c_4$ are
a function of $\theta$. Since by our assumption $\dot{W}_\theta$
is negative definite for $\theta$ arbitrarily small, an argument
identical to the one used in the proof of
Theorem~\ref{thm:no.finite.bound} implies that as
$\theta\rightarrow 0$, $W_\theta$ converges to a nonzero sextic
homogeneous polynomial $W_0$ whose derivative $\dot{W}_0$ along
the trajectories of (\ref{eq:non.monotonicity.vec.field}) (with
$\theta=0$) is non-positive. However, note that when $\theta=0$,
the trajectories of (\ref{eq:non.monotonicity.vec.field}) go on
periodic orbits tracing the level sets of the function $x^4+y^4$.
This implies that
$\dot{W}_0=\frac{\partial{W_0}}{\partial{x}}y^3+\frac{\partial{W_0}}{\partial{y}}(-x^3)=0$.
If we write out this equation, we obtain \scalefont{.92}
\begin{equation}\nonumber
\dot{W}_0=(6c_1-4c_2)x^5y^3+2c_2xy^7-2c_3x^7y+(4c_3-6c_4)x^3y^5=0,
\end{equation}\normalsize
which implies that $c_1=c_2=c_3=c_4=0$, hence a contradiction.
\end{proof}

\begin{remark}
We have numerically computed the range $0<\theta<0.0267$, for
which the conclusion of Theorem~\ref{thm:no.monotonicity} holds.
This bound has been computed via sum of squares relaxation and
semidefinite programming (SDP) by using the SDP solver
SeDuMi~\cite{sedumi}. What allows the search for a Lyapunov
function for the vector field in
(\ref{eq:non.monotonicity.vec.field}) to be exactly cast as a
semidefinite program is the fact that all nonnegative bivariate
forms are sums of squares.
\end{remark}

\section{Acknowledgements}
I would like to express my gratitude to my advisor, Pablo Parrilo,
for his very fruitful comments and suggestions on this paper.

\bibliographystyle{unsrt}
\bibliography{pablo_amirali}

\end{document}